\documentclass[12pt,reqno]{amsproc}%
\usepackage{amsfonts}
\usepackage{amsmath}
\usepackage{amssymb}
\usepackage{graphicx}%
\providecommand{\U}[1]{\protect\rule{.1in}{.1in}}
\theoremstyle{plain}
 \theoremstyle{plain}
\newtheorem{theorem}{Theorem}

\newtheorem{corollary}[theorem]{Corollary}

\newtheorem{definition}[theorem]{Definition}

\newtheorem{lemma}[theorem]{Lemma}

\newtheorem{remark}[theorem]{Remark}

 \numberwithin{equation}{section}
\begin{document}
\title[Modified Similarity Degree]{A Modified Similarity Degree for C*-algebras}
\author{Don Hadwin}
\address{Don Hadwin \\
        Demartment of Mathematics and Statistics \\
         University of New Hampshire\\
         Durham, NH 03824;   Email: don@unh.edu}
\author{Junhao Shen}
\address{Junhao Shen \\
        Demartment of Mathematics and Statistics \\
         University of New Hampshire\\
         Durham, NH 03824;   Email:  Junhao.Shen@unh.edu}
\thanks{Supported by a grant from the National Science Foundation}

\begin{abstract}
We define variants of Pisier's similarity degree for unital
C*-algebras and use direct integral theory to obtain new results. We
prove that if every II$_{1}$ factor representation of a separable
C*-algebra $\mathcal{A}$ has property $\Gamma$, then the similarity
degree of $\mathcal{A}$ is at most $11$.

\end{abstract}
\maketitle

G. Pisier's \emph{similarity degree} \cite{Pisier 0}-\cite{Pisier 5} has been
one of the most far-reaching advances on R. Kadison's similarity problem
\cite{Kadison1}, which asks whether every bounded homomorphism $\rho$ from a
C*-algebra $\mathcal{A}$ into the operators on a Hilbert space must be similar
to a $\ast$-homomorphism. Many classical results on the similarity degree are
contained in \cite{Pisier 5}. There has been some recent interest on this
subject \cite{F. Pop}, \cite{li}, \cite{JW}, \cite{PTWW}, \cite{HL}.

In this paper we define two variants of G. Pisier's similarity
degree \cite{Pisier 0}-\cite{Pisier 5}, one for C*-algebras and one
for von Neumann algebras. Our main result (Theorem \ref{main})
relates our C*-invariant for a separable unital C*-algebra to the
supremum of the W*-invariant of all the II$_{1}$ factor
representations of the algebra. This result yields bounds on the
similarity degree  in some new cases, including some crossed
products and the class of separable unital C*-algebras whose
II$_{1}$ factor representations all have property $\Gamma$.

It was shown by U. Haagerup \cite{Haagerup 1} (see also the union of \cite{H}
and \cite{W}) that a bounded homomorphism $\rho$ on a C*-algebra is similar to
a $\ast$-homomorphism if and only if it is completely bounded, i.e.,
$\left\Vert \rho\right\Vert _{cb}=\sup_{n\in\mathbb{N}}\left\Vert \rho
_{n}\right\Vert <\infty$, where $\rho_{n}$ is defined on $\mathcal{M}%
_{n}\left(  \mathcal{A}\right)  $ by%
\[
\rho_{n}\left( [ a_{ij}]\right)  =\left[  \rho\left(  a_{ij}\right)
\right]
\]
for every $\left[  a_{ij}\right]  \in\mathcal{M}_{n}\left(  \mathcal{A}%
\right)  $.

G. Pisier defined the \emph{similarity degree} of $\mathcal{A}$, denoted by
$d\left(  \mathcal{A}\right)  $ to be the smallest positive integer $d$ (if
one exists) for which there is a positive number $\kappa$ such that, for every
unital C*-algebra $\mathcal{B}$ and every bounded unital homomorphism
$\rho:\mathcal{A}\rightarrow\mathcal{B}$ we have%
\[
\left\Vert \rho\right\Vert _{cb}\leq\kappa\left\Vert \rho\right\Vert
^{d\left(  \mathcal{A}\right)  }.
\]
If there is no such pair $d,\kappa$, we define $d\left(  \mathcal{A}\right)
=\infty$. We denote the smallest $\kappa$ by $\kappa\left(  \mathcal{A}%
\right)  $.

There has been much attention focused on the similarity degree
without much attention to $\kappa$. It is known \cite{Pisier 5} that
when $\mathcal{A}$ is finite-dimensional, $d\left(
\mathcal{A}\right)  =1$ and that when $\mathcal{A}$ is
infinite-dimensional, then $d\left(  \mathcal{A}\right)  \geq2$. It
follows, for example that, for any strictly increasing sequence
$\left\{ \mathcal{A}_{n}\right\} $ of finite-dimensional
C*-algebras, $d\left(  \mathcal{A}_{n}\right)  =1$ for every
$n\in\mathbb{N}$, but $\kappa\left(  \mathcal{A}_{n}\right)
\rightarrow\infty$, otherwise there would be an infinite-dimensional
direct limit $\mathcal{A}$ with $d\left(  \mathcal{A}\right)  =1$.
Hence determining the similarity degree of a direct sum or direct
limit of C*-algebras with the same similarity degree is difficult,
without also controlling the $\kappa$'s. This motivates us to
introduce a new invariant that incorporates both $d$ and $\kappa$.

\begin{definition}\label{def 1} We define the \emph{modified similarity degree} $\hat{d}\left(  \mathcal{A}%
\right)  $ of $\mathcal{A}$ to be the smallest positive number
$\gamma\geq d\left( \mathcal{A}\right)  $ such that, for every
bounded unital algebra homomorphism
$\rho:\mathcal{A}\rightarrow B\left(  H\right)  $, we have %
\[
\left\Vert \rho\right\Vert _{cb}\leq\gamma\left\Vert \rho\right\Vert ^{\gamma
}.
\]
\end{definition}

\begin{definition} \label{def 2} If $\mathcal{A}$ is a von Neumann algebra, we similarly define $\hat{d}_{\ast
}\left(  \mathcal{A}\right)  $ to be the smallest positive number $\gamma$ such that%
\[
\left\Vert \rho\right\Vert _{cb}\leq\gamma\left\Vert \rho\right\Vert ^{\gamma}%
\]
whenever a bounded unital homomorphism $\rho:\mathcal{A}\rightarrow
B\left( H\right)  $ is   ultra*strong-ultra*strong continuous on the
closed unit ball of $\mathcal{A}$, equivalently,
ultrastrong-ultrastrong continuous
on the closed unit ball of $\mathcal{A}^{sa}=\left\{  \operatorname{Re}%
a:a\in\mathcal{A}\right\}  $.

\end{definition}

The following result is elementary.

\begin{lemma}
Suppose $\mathcal{A}$ is a unital C*-algebra. Then

\begin{enumerate}
\item If $\mathcal{J}$ is a closed $\ast$-ideal in $\mathcal{A}$, then%
\[
\hat{d}\left(  \mathcal{A}/\mathcal{J}\right)  \leq\hat{d}\left(
\mathcal{A}\right)
\]

\item
\[
d\left(  \mathcal{A}\right)  \leq\hat{d}\left(  \mathcal{A}\right)  \leq
\max\left(  d\left(  \mathcal{A}\right)  ,\kappa\left(  \mathcal{A}\right)
\right)  \text{.}%
\]

\item If $\mathcal{A}$ is the norm closure of the union of an increasingly
directed family $\left\{  \mathcal{A}_{\lambda}:\lambda\in\Lambda\right\}  $
of unital C*-algebras, then%
\[
\hat{d}\left(  \mathcal{A}\right)  \leq\liminf_{\lambda}\hat{d}\left(
\mathcal{A}_{\lambda}\right)  \text{.}%
\]

\item If, in statement $\left(  3\right)  $ above, $\mathcal{A}$ is the
weak
operator closure (or, strong operator closure) of the union of the $\mathcal{A}_{\lambda}$'s, then%
\[
\hat{d}_{\ast}\left(  \mathcal{A}\right)  \leq\liminf_{\lambda}\hat{d}_{\ast
}\left(  \mathcal{A}_{\lambda}^{\prime\prime}\right)  \leq\liminf_{\lambda
}\hat{d}\left(  \mathcal{A}_{\lambda}\right)  \text{.}%
\]

\end{enumerate}
\end{lemma}

\begin{proof}
$\left(  1\right)  $ and $\left(  2\right)  $ are obvious.

$\left(  3\right)  $. If $\rho:\mathcal A\rightarrow B(H)$ is a
bounded unital homomorphism, then,  $\forall  \ n \in \mathbb{N}$,%
\[\begin{aligned}
\left\Vert \rho_{n}\right\Vert  =\lim_{\lambda}\left\Vert \rho_{n}%
|_{\mathcal{A}_{\lambda}}\right\Vert &\leq\lim_{\lambda}\left\Vert
\rho _{n}|_{\mathcal{A}_{\lambda}}\right\Vert _{cb}
 \\
&\leq\liminf_{\lambda}\hat{d}\left(  \mathcal{A}_{\lambda}\right)
\left\Vert
\rho|_{\mathcal{A}_{\lambda}}\right\Vert ^{\hat{d}\left(  \mathcal{A}%
_{\lambda}\right)  }\\&\leq\liminf_{\lambda}\hat{d}\left(
\mathcal{A}_{\lambda }\right)  \left\Vert \rho\right\Vert
^{\hat{d}\left( \mathcal{A}_{\lambda}\right)
}%
\\
&\leq\left(  \liminf_{\lambda}\hat{d}\left(
\mathcal{A}_{\lambda}\right) \right)  \left\Vert \rho\right\Vert
^{\liminf_{\lambda}\hat{d}\left( \mathcal{A}_{\lambda}\right)  }.
\end{aligned}
\]
Since $\left\Vert \rho\right\Vert _{cb}=\sup_{n\in\mathbb{N}}\left\Vert
\rho_{n}\right\Vert $, the desired result is proved.

$\left(  4\right)  $. Let $\mathcal{B}$ be the norm closure of the
union of the  $\mathcal{A}_{\lambda}$'s. If
$\rho:\mathcal{A}\rightarrow B\left( H\right)  $ is a unital
homomorphism that is ultrastrong-ultrastrong continuous on the
closed unit ball of $\mathcal{A}^{sa}$, then it follows from the
Kaplansky density theorem that $\left\Vert \rho\right\Vert
=\left\Vert \rho|_{\mathcal{B}}\right\Vert $ and $\left\Vert
\rho\right\Vert _{cb}=\left\Vert \rho|_{\mathcal{B}}\right\Vert
_{cb}$. The rest follows from $\left( 3\right)  $.
\end{proof}

\begin{corollary}
Suppose $\mathcal{M}$ is a von Neumann Algebra. Then%
\[
\hat{d}_{\ast}\left(  \mathcal{M}\right)  \leq\inf\left\{  \hat{d}\left(
\mathcal{A}\right)  :\mathcal{A\subseteq M}\text{, }\mathcal{A}\text{ a
C*-algebra, }\mathcal{A}^{\prime\prime}=\mathcal{M}\right\}  .
\]

\end{corollary}

It was shown by U. Haagerup \cite{Haagerup 1} that if a unital
C*-algebra has no tracial states, then $d\left(  \mathcal{A}\right)
=3$ with $\kappa=1$, which implies $\hat{d}\left( \mathcal{A}\right)
\leq3$. Hence if $\mathcal{M}$ is a type I$_{\infty}$, type
 II$_{\infty}$ or type III factor, then $d\left( \mathcal{M}\right)
=\hat{d}\left(  \mathcal{M}\right)  =3.$ In particular,
$\hat{d}\left(  \mathcal{B}\left(  \ell^{2}\right) \right)  =3$. We
see that $\hat{d}_{\ast}$ does a little better.

\begin{corollary}
If $\mathcal{M}$ is a hyperfinite von Neumann algebra, then
$\hat{d}_{\ast }\left(  \mathcal{M}\right)  \leq2.$
\end{corollary}

\begin{corollary}
Suppose $\mathcal{A}$ is a unital C*-algebra. Then%
\[
\hat{d}\left(  \mathcal{A}\right)  =\hat{d}_{\ast}\left(  \mathcal{A}%
^{\#\#}\right)  .
\]

\end{corollary}

If $\tau$ is a tracial state on a unital C*-algebra $\mathcal{A}$, we let
$L^{2}\left(  \mathcal{A},\tau\right)  $ denote the Hilbert space induced by
the inner product $\left\langle a,b\right\rangle =\tau\left(  b^{\ast
}a\right)  $ and let $\pi_{\tau}:\mathcal{A}\rightarrow B\left(  L^{2}\left(
\mathcal{A},\tau\right)  \right)  $ be the GNS representation on $\mathcal{A}$
defined by%
\[
\pi_{\tau}\left(  a\right)  \left(  b\right)  =ab
\]
whenever $a,b\in\mathcal{A}$. We define
\[
\mathcal{M}_{\tau}\left(  \mathcal{A}\right)  =\pi_{\tau}\left(
\mathcal{A}\right)  ^{\prime\prime}%
\]
be the von Neumann algebra generated by $\pi_{\tau}\left(  \mathcal{A}\right)
$. It is known (e.g., see \cite{HM}) that $\tau$ is an extreme point of the
set of tracial states on $\mathcal{A}$ if and only if $\mathcal{M}_{\tau
}\left(  \mathcal{A}\right)  $ is a finite factor von Neumann algebra, and, in
this case we call $\tau$ a \emph{factor tracial state}.

\begin{definition}\label{def 7} We define the \emph{modified tracial similarity degree} of a
unital C*-algebra $\mathcal{A}$
as%
\[
\hat{d}_{tr}\left(  \mathcal{A}\right)  =\sup\left\{  \hat{d}_{\ast}\left(
\mathcal{M}_{\tau}\left(  \mathcal{A}\right)  \right)  :\tau\text{ is a factor
tracial state of }\mathcal{A}\right\}  \text{.}%
\]\end{definition}

Our main result explicitly shows how finding $\hat{d}$ for separable
C*-algebras reduces to finding $\hat{d}_{\ast}$ for II$_{1}$ factor
von Neumann algebras.

\begin{theorem}
\label{main}Suppose $\mathcal{A}$ is a separable unital C*-algebra. Then,
\[
d\left(  \mathcal{A}\right)  \leq\hat{d}\left(  \mathcal{A}\right)
\leq2+3\max\left(  3,\hat{d}_{tr}\left(  \mathcal{A}\right)  \right)  .
\]
In particular, for every unital bounded homomorphism $\rho
:\mathcal{A\rightarrow}B\left(  \ell^{2}\right)  $%
\[
\left\Vert \rho\right\Vert _{cb}\leq\max\left(  3,\hat{d}_{tr}\left(
\mathcal{A}\right)  \right)  \left\Vert \rho\right\Vert ^{2+3\max\left(
3,\hat{d}_{tr}\left(  \mathcal{A}\right)  \right)  }.
\]

\end{theorem}

\begin{proof}
Suppose $\rho:\mathcal{A}\rightarrow B\left(  H\right)  $ is a unital faithful
$\ast$-homomorphism where $H$ is a separable infinite-dimensional Hilbert
space. By replacing $\rho$ with $\rho^{\left(  \infty\right)  }=\rho\oplus
\rho\oplus\cdots$, we can assume that $\rho$ is unitarily equivalent to
$\rho\oplus\rho\oplus\cdots$ . We can extend $\rho$ to a normal unital
homomorphism $\hat{\rho}:\mathcal{A}^{\#\#}\rightarrow B\left(  H\right)  $
such that $\hat{\rho}|_{\mathcal{A}}=\rho$ and such that $\hat{\rho}$ is
unitarily equivalent to $\hat{\rho}^{\left(  \infty\right)  }$ . Since
$\ker\hat{\rho}$ is a weak*-closed ideal in the von Neumann algebra
$\mathcal{A}^{\#\#},$ there is a central projection $Q\in\mathcal{A}^{\#\#}$
such that $\ker\hat{\rho}=(1-Q)\mathcal{A}^{\#\#}$. Let $\mathcal{M}%
=Q\mathcal{A}^{\#\#}$. Since $\hat{\rho}$ is normal, we see that $\hat{\rho
}\left(  \mathcal{M}\right)  $ is weak*-closed (Krein-Shmulyan) and
$\sigma=\hat{\rho}|_{\mathcal{M}}:\mathcal{M}\rightarrow\hat{\rho}\left(
\mathcal{M}\right)  $ is a weak*-weak* homeomorphism. Thus the predual of
$\mathcal{M}$ is separable, so there is a normal faithful $\ast$-homomorphism
$\pi:\mathcal{M}\rightarrow B\left(  H\right)  $, and we can assume that $\pi$
is unitarily equivalent to $\pi^{\left(  \infty\right)  }$. Hence we can
assume that $\mathcal{A}\subseteq\mathcal{M}\subseteq B\left(  H\right)  $,
$\mathcal{M}=\mathcal{A}^{\prime\prime},$ $id_{\mathcal{M}}$ is unitarily
equivalent to $id_{\mathcal{M}}^{\left(  \infty\right)  }$ and that
$\sigma:\mathcal{M\rightarrow}$ $B\left(  H\right)  $ is a faithful normal
homomorphism such that $\sigma|_{\mathcal{A}}=\rho$ and $\sigma$ is unitarily
equivalent to $\sigma^{\left(  \infty\right)  }$. It follows from the
Pisier-Ringrose theorem \cite{Pisier 0}, \cite{R} that $\sigma$ is
ultra*strong-ultra*strong continuous. Since $\rho$ and $\pi$ have infinite
multiplicity, we know that $\sigma$ is SOT-SOT continuous on $\mathcal{M}%
^{sa}$.

It follows from direct integral theory that, up to unitary equivalence, there
is a separable infinite-dimensional Hilbert space $K$ and a probability
measure space $\left(  \Omega,\mu\right)  $ such that $H=L^{2}\left(
\mu,K\right)  $ and the center $\mathcal{Z}\left(  \mathcal{M}\right)  $ of
$\mathcal{M}$ is
\[
\left\{  \varphi\in L^{\infty}\left(  \mu,B\left(  H\right)  \right)
:\varphi\left(  \omega\right)  \in\mathbb{C}1\text{ a.e.}\right\}  \text{.}%
\]
Hence $\mathcal{M}\subseteq\mathcal{Z}\left(  \mathcal{M}\right)
^{\prime }=L^{\infty}\left(  \mu,B\left(  H\right)  \right)  $ and
if $\left\{ \varphi_{1},\varphi_{2},\ldots\right\}  $ is a
selfadjoint norm-dense subset of the closed unit ball of
$\mathcal{A}^{sa}$, and hence a strong-operator-dense subset of the
closed unit ball of $\mathcal{M}^{sa}$, and, for each
$\omega\in\Omega$, we let $\mathcal{M}_{\omega}=\left\{
\varphi_{1}\left(  \omega\right)
,\varphi_{2}\left(  \omega\right)  ,\ldots\right\}  ^{\prime\prime}$, we have%
\[
\mathcal{M}=\int_{\Omega}^{\oplus}\mathcal{M}_{\omega}d\mu\left(
\omega\right)  ,
\]
and $\mathcal M$ is  \ $\left\{  \varphi\in L^{\infty}\left(
\mu,B\left( H\right) \right)  :\varphi\left(  \omega\right)
\in\mathcal{M}_{\omega}\text{ }a.e\right\}  $. Moreover, \  we can
further assume  that $\left\{  \varphi_{1}\left( \omega\right)
,\varphi_{2}\left(  \omega\right)  ,\ldots\right\}  $ is
strong-operator-dense in the closed unit ball of
$\mathcal{M}_{\omega}^{sa}$ a.e., and that there is a measurable
family $\left\{ \pi_{\omega}:\omega\in\Omega\right\}  $ of unital
$\ast$-homomorphisms from $\mathcal{A}$ to $B\left(  K\right)  $
such that,
\[
id_{\mathcal{A}}=\int_{\Omega}^{\oplus}\pi_{\omega}d\mu
\]
and%
\[
\pi_{\omega}\left(  \mathcal{A}\right)  ^{\prime\prime}=\mathcal{M}_{\omega
}\text{ a.e. .}%
\]

Now we want to look at the restriction of $\sigma$ to
$\mathcal{Z}\left( \mathcal{M}\right)  $, which is a bounded unital
normal injective homomorphism with $\left\Vert \rho\right\Vert
=\left\Vert \sigma\right\Vert $. Since $d\left(  \mathcal{Z}\left(
\mathcal{M}\right)  \right)  =2$, there is an invertible operator
$S\in B\left(  H\right)  $ such that $\left\Vert S\right\Vert
\left\Vert S^{-1}\right\Vert \leq\left\Vert \rho\right\Vert ^{2}$
and such that $\sigma_{1}\left(  \cdot \right)  =S\sigma\left(
\cdot\right) S^{-1}$ is an injective normal $\ast$-homomorphism on
$\mathcal{Z}\left( \mathcal{M}\right)  $. Hence
\[
\left\Vert \sigma_{1}\right\Vert \leq\left\Vert \rho\right\Vert ^{2}\left\Vert
\sigma\right\Vert =\left\Vert \rho\right\Vert ^{3}\text{.}%
\]
Moreover, $\sigma\left(  \cdot\right)  =S\sigma_{1}\left(  \cdot\right)
S^{-1},$ so
\[
\left\Vert \rho\right\Vert _{cb}=\left\Vert \sigma\right\Vert _{cb}%
\leq\left\Vert \rho\right\Vert ^{2}\left\Vert \sigma_{1}\right\Vert _{cb}.
\]
Since $\sigma_{1}$ is unitarily equivalent to $\sigma_{1}^{\left(
\infty\right)  }$ and $id_{\mathcal{M}}$ is unitarily equivalent to $\left(
id_{\mathcal{M}}\right)  ^{\left(  \infty\right)  },$ we see that $\sigma
_{1}|\mathcal{Z}\left(  \mathcal{M}\right)  $ is unitarily equivalent to
$id_{\mathcal{Z}\left(  \mathcal{M}\right)  }$. By putting this unitary with
$S,$ we can assume that $\sigma_{1}\left(  T\right)  =T$ for every
$T\in\mathcal{Z}\left(  \mathcal{M}\right)  $. This means that $\sigma_{1}$ is
an $L^{\infty}\left(  \mu\right)  =\mathcal{Z}\left(  \mathcal{M}\right)  $
module homomorphism.

Since $\mathcal{M}\subseteq\mathcal{Z}\left(  \mathcal{M}\right)  ^{\prime}$,
we know that
\[
\sigma_{1}\left(  \mathcal{M}\right)  \subseteq\mathcal{Z}\left(
\mathcal{M}\right)  ^{\prime}=L^{\infty}\left(  \mu,B\left(  K\right)
\right)  \text{.}%
\]
Hence we can find functions $\psi_{1},\psi_{2},\ldots\in L^{\infty}\left(
\mu,B\left(  K\right)  \right)  $ such that
\[
\sigma_{1}\left(  \varphi_{n}\right)  =\psi_{n}%
\]
for every $n\geq1.$

It is important to note that the $\varphi_{n}$'s and $\psi_{n}$'s
are actually representatives of equivalence classes since we
identify functions that agree almost everywhere. So we must now take
some care with sets of measure $0$. First note that if $p\left(
t_{1},\ldots,t_{n}\right)  $ is a $\ast $-polynomial, then
$$\sigma_{1}\left(  p\left(  \varphi_{1},\ldots,\varphi _{n}\right)
\right)  =p\left(  \psi_{1},\ldots,\psi_{n}\right)  ,$$ and
$$\left\Vert p\left(  \psi_{1},\ldots,\psi_{n}\right)  \right\Vert
_{\infty }\leq\left\Vert \sigma_1\right\Vert \left\Vert p\left(
\varphi_{1},\ldots ,\varphi_{n}\right)  \right\Vert _{\infty}.$$ We
want to get a better estimate.
For each $k\geq1$, let $E_{k}$ be the set of all $\omega\in\Omega$ such that%
\[
\left\Vert p\left(  \psi_{1}\left(  \omega\right)  ,\ldots,\psi_{n}\left(
\omega\right)  \right)  \right\Vert >\left\Vert \sigma_{1}\right\Vert
\left\Vert p\left(  \varphi_{1}\left(  \omega\right)  ,\ldots,\varphi
_{n}\left(  \omega\right)  \right)  \right\Vert +\frac{1}{k}.
\]
Since $\sigma_{1}$ is an $L^{\infty}\left(  \mu\right)
=\mathcal{Z}\left(  \mathcal{M}\right)  $ module homomorphism, we
have $$\sigma_{1}\left(  \chi_{E_{k}}p\left(
\varphi_{1},\ldots,\varphi _{n}\right)  \right) =\chi_{E_{k}}p\left(
\psi_{1},\ldots,\psi_{n}\right) ,$$ and $$\left\Vert
\chi_{E_{k}}p\left(  \psi_{1},\ldots,\psi_{n}\right) \right\Vert
_{\infty}\leq\left\Vert \sigma_{1}\right\Vert \left\Vert
\chi_{E_{k}}p\left(  \varphi_{1},\ldots,\varphi_{n}\right)
\right\Vert _{\infty}.$$ It follows that $\mu\left(  E_{k}\right)
=0.$ Hence $\mu\left(
\cup_{n\geq1}\mu\left(  E_{k}\right)  \right)  =0.$ Therefore%
\[
\left\Vert p\left(  \psi_{1}\left(  \omega\right)  ,\ldots,\psi_{n}\left(
\omega\right)  \right)  \right\Vert \leq\left\Vert \sigma_{1}\right\Vert
\left\Vert p\left(  \varphi_{1}\left(  \omega\right)  ,\ldots,\varphi
_{n}\left(  \omega\right)  \right)  \right\Vert \text{ }a.e.
\] Now we have   the following claim.

{\em Claim 8.1: Let $\mathcal{P}$ denote the set of a
$\ast$-polynomials with coefficients in
$\mathbb{C}_{\mathbb{Q}}=\mathbb{Q}+i\mathbb{Q}$. We then have%
\[
\left\Vert p\left(  \psi_{1}\left(  \omega\right)  ,\ldots,\psi_{n}\left(
\omega\right)  \right)  \right\Vert \leq\left\Vert \sigma_{1}\right\Vert
\left\Vert p\left(  \varphi_{1}\left(  \omega\right)  ,\ldots,\varphi
_{n}\left(  \omega\right)  \right)  \right\Vert \text{ }a.e.
\]
for every $p\in\mathcal{P}$. By removing a set of measure zero from
$\Omega$, we can assume that the above relation hold for every
$\omega\in\Omega$.}

Let $\left\{  1=\gamma_{1},\gamma_{2},\ldots\right\}  $ be an orthonormal
basis for $L^{2}\left(  \mu\right)  $ and let $\left\{  e_{1},e_{2}%
,\ldots\right\}  $ be an orthonormal basis for $K$. Define $u_{i,k}\in H$ by%
\[
u_{i,k}\left(  \omega\right)  =\gamma_{i}\left(  \omega\right)
e_{k}, \qquad \forall \ i,k\in \Bbb N.
\]
It is well known that $\left\{  u_{i,k}:i,k\in\mathbb{N}\right\}  $
is an orthonormal basis for $H=L^{2}\left(  \mu,K\right)  $. We can
define a metric
$d_{H}$ on $B\left(  H\right)  $ by%
\[
d_{H}\left(  S,T\right)  ^{2}=\sum_{i,k=1}^{\infty}\frac{\left\Vert
\left( S-T\right)  u_{i,k}\right\Vert ^{2}}{2^{i+k}}, \quad \forall
\ S, T\in B(H)
\]
and a metric $d_{K}$ on $B\left(  K\right)  $ by%
\[
d_{K}\left(  S,T\right)  ^{2}=\sum_{k=1}^{\infty}\frac{\left\Vert
\left( S-T\right)  e_{k}\right\Vert ^{2}}{2^{k}},  \quad \forall \
S, T\in B(K).
\]
On bounded subsets of $B\left(  H\right)  $ (respectively, $B\left(  K\right)
$) the metric $d_{H}$ (respectively, $d_{K}$) induces the strong operator topology.

We know that $\sigma:\mathcal M\rightarrow B(H)$ is
ultrastrongly-ultrastrongly continuous on $\mathcal M^{sa}$. It
follows that $\sigma_{1}$ is uniformly continuous on the closed unit
ball $\mathcal{B}$ of $\mathcal{M}^{sa}$, since if $\left\{
S_{n}\right\} ,\left\{ T_{n}\right\}  $ are sequences in
$\mathcal{B}$, we know $d_{H}\left( S_{n},T_{n}\right)
\rightarrow0$ if and only if $S_{n}-T_{n}\rightarrow0$
ultrastrongly, implying $\sigma_{1}\left( S_{n}-T_{n}\right)
=\sigma _{1}\left( S_{n}\right) -\sigma_{1}\left(  T_{n}\right)
\rightarrow0$ ultrastrongly, which implies $d_{H}\left(
S_{n},T_{n}\right) \rightarrow0$. Therefore, we have  the following
claim.

{\em Claim 8.2: Suppose $s,r\in\mathbb{N}$. Then there is a
$t_{r,s}\in\mathbb{N}$ such that,
for every $S,T\in\mathcal{B}$, we have%
\[
d_{H}\left(  S,T\right)  <\frac{1}{t_{r,s}}\Rightarrow d_{H}\left(
\sigma_{1}\left(  S\right)  ,\sigma_{1}\left(  T\right)  \right)
<\frac {1}{rs}.
\]}
Let $\mathcal{P}_{1}$ be the set of $p\in\mathcal{P}$ such that $\left\Vert
p\left(  \psi_{1},\psi_{2},\ldots\right)  \right\Vert \leq1$ and $p\left(  \psi_{1},\psi_{2},\ldots\right) =p\left(  \psi_{1},\psi_{2},\ldots\right) ^*$, and write%
\[
\mathcal{P}_{1}\times\mathcal{P}_{1}=\left\{  \left(  p_{1},q_{1}\right)
,\left(  p_{2},q_{2}\right)  ,\ldots\right\}  \text{.}
\]
Let $r,s$ be in $ \mathbb{N}$ and $t_{r,s}$ be as in Claim 8.2. Now
suppose $ j\in\mathbb{N}$ , and let $E_{j,r,s}$ denote the set of
all
$\omega\in\Omega$ such that%
\[
d_{K}\left(  p_{j}\left(  \psi_{1}\left(  \omega\right)  ,\ldots,\psi
_{n}\left(  \omega\right)  \right)  ,q_{j}\left(  \psi_{1}\left(
\omega\right)  ,\ldots,\psi_{n}\left(  \omega\right)  \right)  \right)
^{2}<\frac{1}{t_{r,s}}\text{,}
\]
and%
\[
d_{K}\left(  p_{j}\left(  \varphi_{1}\left(  \omega\right)  ,\ldots
,\varphi_{n}\left(  \omega\right)  \right)  ,q_{j}\left(  \varphi_{1}\left(
\omega\right)  ,\ldots,\varphi_{n}\left(  \omega\right)  \right)  \right)
^{2}\geq\frac{1}{s}.
\]
Let $F_{1,r,s}=E_{1,r,s}$ and let $F_{j+1,r,s}=E_{j+1,r,s}\backslash
\cup_{1\leq i\leq j}E_{i,r,s}$ for $j\in\mathbb{N}$. Hence if $S=\sum
_{j=1}^{\infty}\chi_{F_{j,r,s}}p_{j}\left(  \psi_{1},\ldots,\psi_{n}\right)  $
and $T=\sum_{j=1}^{\infty}\chi_{F_{j,r,s}}q_{j}\left(  \psi_{1},\ldots
,\psi_{n}\right)  ,$ we have, for $ i, k \in\mathbb{N}$, that%
\[\begin{aligned}
&\left\Vert \left(  S-T\right)  u_{i,k}\right\Vert ^{2}
\\&\ \ =
\sum_{j=1}^{\infty}\int_{F_{j,r,s}}\left\vert \gamma_{i}\left(
\omega\right) \right\vert ^{2}\left\Vert \left(  p_{j}\left(
\psi_{1}\left(  \omega\right) ,\ldots,\psi_{n}\left(  \omega\right)
\right)  -q_{j}\left(  \psi_{1}\left( \omega\right)
,\ldots,\psi_{n}\left(  \omega\right)  \right)  \right)
e_{k}\right\Vert ^{2}d\mu \end{aligned}
\]and \[\begin{aligned}
d_{H}&\left(  S,T\right)  ^{2}
=\sum_{i,k=1}^{\infty}\frac{\left\Vert \left( S-T\right)
u_{i,k}\right\Vert ^{2}}{2^{i+k}}\\ & =\sum_{i,j=1}^\infty
 \int_{F_{j,r,s}}\left
(\sum_{k=1}^\infty \frac {\left\vert \gamma_{i} \left( \omega\right)
\right\vert ^{2}\left\Vert \left( p_{j}\left( \psi_{1}\left(
\omega\right) ,\ldots,\psi_{n}\left( \omega\right) \right)
-q_{j}\left( \psi_{1}\left( \omega\right) ,\ldots,\psi_{n}\left(
\omega\right) \right)  \right) e_{k}\right\Vert ^{2}}
{2^{i+k}}\right )d\mu
\\ &
<\sum_{i=1}^\infty\frac{1}{t_{r,s}}\frac 1
{2^{i}}\int_{\Omega}\left\vert \gamma_{i}\left( \omega\right)
\right\vert ^{2}d\mu =\frac{1}{t_{r,s}} .\end{aligned}
\]
  Thus, by Claim 8.2, we have  $d_{H}\left(
\sigma_{1}\left(  S\right)  ,\sigma_{1}\left(  T\right)  \right)
<\frac {1}{rs}$. However, since $\gamma_1(\omega)=1$, we see that%

\[\begin{aligned}
 &\left\Vert \left(  \sigma_{1}\left(  S\right)  -\sigma_{1}\left(
T\right) \right)  u_{1,k}\right\Vert ^{2}=
\\& \qquad \qquad
\sum_{j=1}^{\infty}\int_{F_{j,r,s}}\left\Vert \left(  p_{j}\left(
\varphi _{1}\left(  \omega\right)  ,\ldots,\varphi_{n}\left(
\omega\right)  \right) -q_{j}\left(  \psi_{1}\left(  \omega\right)
,\ldots,\psi_{n}\left( \omega\right)  \right)  \right)
e_{k}\right\Vert ^{2}d\mu, \end{aligned}
\]
which implies that%
\[\begin{aligned}
\frac{1}{rs} &> d_{H}\left(  \sigma_{1}\left(  S\right)
,\sigma_{1}\left( T\right)  \right)
\\&  \geq
\sum_{k=1}^{\infty}\frac{\left\Vert \left(  \sigma_{1}\left(
S\right) -\sigma_{1}\left(  T\right)  \right)  u_{1,k}\right\Vert
^{2}}{2^{1+k}}
\\&  \geq
\frac{1}{2}\sum_{j=1}^{\infty}\int_{F_{j,r,s}}\sum_{k=1}^{\infty}%
\frac{\left\Vert \left(  p_{j}\left(  \varphi_{1}\left(  \omega\right)
,\ldots,\varphi_{n}\left(  \omega\right)  \right)  -q_{j}\left(  \psi
_{1}\left(  \omega\right)  ,\ldots,\psi_{n}\left(  \omega\right)  \right)
\right)  e_{k}\right\Vert ^{2}}{2^{k}}d\mu
\\&
=\frac{1}{2}\sum_{j=1}^{\infty}\int_{F_{j,r,s}}d_{K}\left(
p_{j}\left( \varphi _{1}\left(  \omega\right)
,\ldots,\varphi_{n}\left( \omega\right)  \right) ,q_{j}\left(
\psi_{1}\left(  \omega\right) ,\ldots,\psi_{n}\left( \omega\right)
\right)  \right)  ^{2}d\mu
\\&
\geq\frac{1}{2}\sum_{j=1}^{\infty}\int_{F_{j,r,s}}\frac{1}{s}d\mu=\frac{1}{2s}\mu\left(
\cup_{j=1}^{\infty}F_{j,r,s}\right) \end{aligned}
\]
Hence%
\[
\frac{2}{r}>\mu\left(  \cup_{j=1}^{\infty}F_{j,r,s}\right)
=\mu\left(
\cup_{j=1}^{\infty}E_{j,r,s}\right)  \text{.}%
\]
Let $W_{r,s}=\cup_{j=1}^{\infty}E_{j,r,s}$. Then $W_{r,s}$ is
precisely the set of all
$\omega\in\Omega$ for which there is a $\left(  p,q\right)  \in\mathcal{P}%
_{1}\times\mathcal{P}_{1}$ such that%
\[
   d_{K}\left(  p\left(  \psi_{1}\left(  \omega\right)  ,\ldots
,\psi_{n}\left(  \omega\right)  \right)  ,q\left(  \psi_{1}\left(
\omega\right)  ,\ldots,\psi_{n}\left(  \omega\right)  \right)  \right)
^{2}   <\frac{1}{t_{r,s}}%
\]
and
\[
d_{K}\left(  p_{j}\left(  \varphi_{1}\left(  \omega\right)  ,\ldots
,\varphi_{n}\left(  \omega\right)  \right)  ,q_{j}\left(  \varphi_{1}\left(
\omega\right)  ,\ldots,\varphi_{n}\left(  \omega\right)  \right)  \right)
^{2}\geq\frac{1}{s}.
\]
It follows that%
\[
\mu\left(  \cup_{s=1}^{\infty}\cap_{r=1}^{\infty}W_{r,s}\right)  =0.
\]
Hence, if we throw away another set of measure $0$, we can assume that%
\[
\cup_{s=1}^{\infty}\cap_{r=1}^{\infty}W_{r,s}=\emptyset\text{,}%
\]
which implies that%
\[
\cap_{s=1}^{\infty}\cup_{r=1}^{\infty}\left(  \Omega\backslash W_{r,s}\right)
=\Omega.
\]
This means that

{\em Claim 8.3: for every $\omega\in\Omega$ and every
$s\in\mathbb{N}$ there are an $r\in\mathbb{N}$ and, thus, a
$t_{r,s}\in\mathbb{N}$ ({as  in Claim 8.2}) such that $\omega\notin
W_{r,s}$, i.e., for every $\left(  p,q\right)
\in\mathcal{P}_{1}\times\mathcal{P}_{1}$,
\[
d_{K}\left(  p\left(  \psi_{1}\left(  \omega\right)  ,\ldots,\psi_{n}\left(
\omega\right)  \right)  ,q\left(  \psi_{1}\left(  \omega\right)  ,\ldots
,\psi_{n}\left(  \omega\right)  \right)  \right)  ^{2}<\frac{1}{t_{r,s}}%
\]
implies that%
\[
d_{K}\left(  p\left(  \psi_{1}\left(  \omega\right)  ,\ldots,\psi_{n}\left(
\omega\right)  \right)  ,q\left(  \psi_{1}\left(  \omega\right)  ,\ldots
,\psi_{n}\left(  \omega\right)  \right)  \right)  ^{2}<\frac{1}{s}.
\]}
It follows from Claim 8.1 that, for each $\omega\in\Omega$ and each
$p\in\mathcal{P}$,
\[
\alpha_{\omega}\left(  p\left(  \psi_{1}\left(  \omega\right)  ,\psi
_{2}\left(  \omega\right)  ,\ldots\right)  \right)  =p\left(  \varphi
_{1}\left(  \omega\right)  ,\varphi_{2}\left(  \omega\right)  ,\ldots\right)
\]
defines a unital algebra homomorphism $\alpha_{\omega}$ from%
\[
\left\{  p\left(  \psi_{1}\left(  \omega\right)  ,\psi_{2}\left(
\omega\right)  ,\ldots\right)  :p\in\mathcal{P}\right\}
\]
to%
\[
\left\{  p\left(  \varphi_{1}\left(  \omega\right)  ,\varphi_{2}\left(
\omega\right)  ,\ldots\right)  :p\in\mathcal{P}\right\}
\]
with $\left\Vert \alpha_{\omega}\right\Vert \leq\left\Vert \sigma
_{1}\right\Vert \leq\left\Vert \rho\right\Vert ^{3}$. The Claim 8.3
shows that, for each $\omega\in\Omega$, $\alpha_{\omega}$ is
$d_{K}$-$d_{K}$ uniformly continuous on $\mathcal P_1$. Hence
$\alpha_{\omega}$ uniquely extends to a unital algebra homomorphism
from $\mathcal{M}_{\omega}$ that is ultrastrong-ultrastrong
continuous on the unit ball of $\mathcal{M}_{\omega }^{sa}$. It
follows that if $\psi\in L^{\infty}\left(  \mu,B\left(  K\right)
\right)  $ is in $\mathcal{M}$ and $\sigma_{1}\left(  \psi\right)
=\varphi\in L^{\infty}\left(  \mu,B\left(  K\right)  \right)  $,
then, for every
$\omega\in\Omega$,%
\[
\varphi\left(  \omega\right)  =\alpha_{\omega}\left(  \psi\left(
\omega\right)  \right)  \text{.}%
\]
It follows that%
\[
\left\Vert \sigma_{1}\right\Vert _{cb}\leq\sup_{\omega\in\Omega}\left\Vert
\alpha_{\omega}\right\Vert _{cb}\text{.}%
\]
If the factor $\mathcal{M}_{\omega}$ has type  I , then it it
hyperfinite, so
$\hat{d}_{\ast}\left(  \mathcal{M}_{\omega}\right)  \leq2$. If $\mathcal{M}%
_{\omega}$ has type II$_{\infty}$ or type  III, then U. Haagerup's
result \cite{Haagerup 1} implies $\hat{d}_{\ast}\left(
\mathcal{M}_{\omega}\right) \leq3$. If the factor
$\mathcal{M}_{\omega}$ has type II$_1$, then  $\hat{d}_{\ast}\left(
\mathcal{M}_{\omega}\right)  \leq \hat{d}_{tr}\left(
\mathcal{A}\right)   $ by Definition \ref{def 7}. Hence,
\[
\sup_{\omega\in\Omega}\hat{d}_{\ast}\left(  \mathcal{M}_{\omega}\right)
\leq\max\left(  3,\hat{d}_{tr}\left(  \mathcal{A}\right)  \right)  .
\]
Therefore%
\[\begin{aligned}
\left\Vert \sigma_{1}\right\Vert
_{cb}&\leq\sup_{\omega\in\Omega}\max\left( 3,\hat{d}_{tr}\left(
\mathcal{A}\right)  \right)  \left\Vert \alpha _{w}\right\Vert
^{\max\left(  3,\hat{d}_{tr}\left(  \mathcal{A}\right) \right)
}&\leq\max\left(  3,\hat{d}_{tr}\left(  \mathcal{A}\right)  \right)
\left\Vert \rho\right\Vert ^{3\max\left(  3,\hat{d}_{tr}\left(
\mathcal{A}\right) \right)  }\text{,} \end{aligned}
\] and
\[\begin{aligned}
\left\Vert \rho\right\Vert _{cb}&=\left\Vert \sigma\right\Vert _{cb}%
\leq\left\Vert \sigma_{1}\right\Vert _{cb}\left\Vert \rho\right\Vert ^{2}%
 \leq3\max\left(  3,\hat{d}_{tr}\left(  \mathcal{A}\right)  \right)
\left\Vert \rho\right\Vert ^{2+3\max\left(  3,\hat{d}_{tr}\left(
\mathcal{A}\right) \right)  }\text{.} \end{aligned}
\]
It follows that%
\[
d\left(  \mathcal{A}\right)  \leq\hat{d}\left(  \mathcal{A}\right)
\leq2+3\max\left(  3,\hat{d}_{tr}\left(  \mathcal{A}\right)  \right)  \text{.}%
\]

\end{proof}

Erik Christensen \cite{Christensen2}  proved that if $\mathcal{M}$
is a II$_{1}$ factor with property $\Gamma$, then $d\left(
\mathcal{M}\right)  =3$ and $\kappa\left( \mathcal{M}\right)  =1,$
which means $\hat{d}\left( \mathcal{M}\right)  \leq3$. We will not
name or try to give an internal description of the class of
C*-algebras in the following Corollary, but we note that it clearly
contains the class of weakly approximately divisible C*-algebras
defined in \cite{HL}.

\begin{corollary}
Suppose $\mathcal{A}$ is a separable unital C*-algebra with that property that
$\pi_{\tau}\left(  \mathcal{A}\right)  ^{\prime\prime}$ has property $\Gamma$
for every factor tracial state $\tau$ on $\mathcal{A}$. Then%
\[
d\left(  \mathcal{A}\right)  \leq\hat{d}\left(  \mathcal{A}\right)
\leq11\text{.}%
\]

\end{corollary}

Note that the following corollary relates $d\left(  C^{\ast}\left(
\mathbb{F}_{n}\right)  \right)  $ and $\hat{d}_{\ast}\left(  \mathcal{L}%
_{\mathbb{F}_{n}}\right)  $ for each integer $n\geq2$.

\begin{corollary}
If $\mathcal{A}$ is a separable unital C*-algebra with a unique tracial state
$\tau$, then%
\[
d\left(  \mathcal{M}_{\tau}\left(  \mathcal{A}\right)  \right)  \leq\hat
{d}_{\ast}\left(  \mathcal{M}_{\tau}\left(  \mathcal{A}\right)  \right)  \leq
d\left(  \mathcal{A}\right)  \leq2+3\max\left(  3,\hat{d}_{\ast}\left(
\mathcal{M}_{\tau}\left(  \mathcal{A}\right)  \right)  \right)  \text{.}%
\]

\end{corollary}

We now apply our results to certain crossed products.

Suppose $n\geq2$ is a positive integer and let $\mathbb{G}_{n}$ denote the
group generated by the $n\times n$ diagonal unitary matrices and the
permutation matrices. We define an action $\alpha$ on $C_{r}^{\ast}\left(
\mathbb{F}_{n}\right)  $ with standard unitary generators $u_{1},\ldots,u_{n}%
$, so that if $g=DV_{\sigma}$ with $D=diag\left(  \lambda_{1},\ldots
,\lambda_{n}\right)  $ and $V_{\sigma}$ is the permutation matrix
corresponding to $\sigma\in S_{n}$, we let $\alpha\left(  g\right)  $ be the
automorphism of $C_{r}^{\ast}\left(  \mathbb{F}_{n}\right)  $ that sends
$u_{k}$ to $\lambda_{k}u_{\sigma\left(  k\right)  }$.

\begin{corollary}
Suppose $n\geq2$ and $H$ is an abelian subgroup of $\mathbb{G}_{n}$ that is
not torsion (i.e., $H$ contains at least one element of infinite order). Then%
\[
\hat{d}_{tr}\left(  C_{r}^{\ast}\left(  \mathbb{F}_{n}\right)  \rtimes
_{\alpha}H\right)  \leq3,
\]
and%
\[
d\left(  C_{r}^{\ast}\left(  \mathbb{F}_{n}\right)  \rtimes_{\alpha}H\right)
\leq11.
\]

\end{corollary}

\begin{proof}
Suppose $h\in H$ has infinite order. Then $h^{n!}\neq I_{n}$ (the $n\times n$
identity matrix, which is the identity element of $\mathbb{G}_{n}$). However,
$h^{n!}$ must be a diagonal matrix $D=diag\left(  \lambda_{1},\ldots
,\lambda_{n}\right)  $. Moreover, $h^{n!}$ has infinite order, so
$\lambda_{j_{0}}$ must be irrational for some $1\leq j_{0}\leq n$. We also
know that there is an increasing sequence $\left\{  m_{k}\right\}  $ in
$\mathbb{N}$ such that $D^{m_{k}}\rightarrow I_{n}$, i.e., $\lim
_{k\rightarrow\infty}\lambda_{j}^{m_{k}}=1$ for $1\leq j\leq n$.

Let $W=W_{h^{n!}}$ be the unitary in $C_{r}^{\ast}\left(  \mathbb{F}%
_{n}\right)  \rtimes_{\alpha}H$ corresponding to $h^{n!},$ so that, for every
$A\in C_{r}^{\ast}\left(  \mathbb{F}_{n}\right)  $,%
\[
WAW^{\ast}=\alpha\left(  h^{n!}\right)  \left(  A\right)  \text{.}%
\]
Since%
\[
\left\Vert W^{m_{k}}u_{j}\left(  W^{m_{k}}\right)  ^{\ast}-u_{j}\right\Vert
=\left\vert \lambda_{j}^{m_{k}}-1\right\vert \left\Vert u_{j}\right\Vert
\rightarrow0
\]
and, since $H$ is abelian $WW_{g}=W_{g}W$ for every $g\in H$. Hence%
\[
\left\Vert W^{m_{k}}T-TW^{m_{k}}\right\Vert \rightarrow0
\]
for every $T\in C_{r}^{\ast}\left(  \mathbb{F}_{n}\right)  \rtimes_{\alpha}H$.

Next suppose $\tau$ is a tracial state on $C_{r}^{\ast}\left(  \mathbb{F}%
_{n}\right)  \rtimes_{\alpha}H$ and $m$ is a positive integer. Since%
\[
W^{m}u_{j_{0}}=\lambda_{j_{0}}^{m}u_{j_{0}}W^{m},
\]
we conclude that $\tau\left(  W^{m}\right)  =0.$ It follows that
$\mathcal{M}_{\tau}\left(  C_{r}^{\ast}\left(  \mathbb{F}_{n}\right)
\rtimes_{\alpha}H\right)  $ is a von Neumann algebra  with property $\Gamma$,
so by \cite{Christensen2},
\[
\hat{d}_{\ast}\left(  \mathcal{M}_{\tau}\left(  C_{r}^{\ast}\left(
\mathbb{F}_{n}\right)  \rtimes_{\alpha}H\right)  ^{\prime\prime}\right)
\leq3.
\]
Hence%
\[
\hat{d}_{tr}\left(  C_{r}^{\ast}\left(  \mathbb{F}_{n}\right)  \rtimes
_{\alpha}H\right)  \leq3,
\]
and%
\[
d\left(  C_{r}^{\ast}\left(  \mathbb{F}_{n}\right)  \rtimes_{\alpha}H\right)
\leq11.
\]

\end{proof}

\begin{remark} \label{remark 1}{\em We conclude by reminding the reader of the equivalent
formulations of Kadison's similarity problem (see, e.g.,
\cite{Pisier 5}) so, for example, showing $d\left(
\mathcal{A}\right)  <\infty$ implies that if $\pi
:\mathcal{A}\rightarrow B\left(  H\right)  $ is a unital
$\ast$-homomorphism and $T\in B\left(  H\right)  $ and $T\left(
H\right)  $ is invariant for every operator in $\pi\left(
\mathcal{A}\right)  $, then there is an $S\in \pi\left(
\mathcal{A}\right)  ^{\prime}$ such that $S\left(  H\right) =T\left(
H\right)  .$ Moreover, there is a $K\geq1$ such that, for every
$W\in B\left(  H\right)  $,%
\[
dist\left(  W,\pi\left(  \mathcal{A}\right)  ^{\prime}\right)  \leq
K\sup\left\{  \left\Vert \pi\left(  U\right)  W-W\pi\left(  U\right)
\right\Vert :U\in\mathcal{A},U\text{ unitary}\right\}  \text{.}%
\]}

\end{remark}

\end{document}